\documentclass[a4paper,12pt,openany,reqno]{amsart}
	\usepackage[latin1]{inputenc}
    \usepackage[T1]{fontenc}
    \usepackage[english]{babel}
    \usepackage{appendix}
    \usepackage[english]{minitoc}
    \usepackage[nice]{nicefrac}
    \usepackage[PetersLenny]{fncychap}
\usepackage[top=2.7cm, bottom=2.7cm, left=2.7cm, right=2.7cm]{geometry}
\numberwithin{equation}{section}  \makeatletter\@addtoreset{equation}{section}
	\usepackage{amsmath}
    \usepackage{amssymb}
    \usepackage{amsbsy} %\usepackage{amsthm}
    \usepackage{latexsym, amsfonts}
    \usepackage{graphics}
    \usepackage{ulem}
    \usepackage{hhline}
    \usepackage{dsfont}
    \usepackage{mathrsfs}
    \usepackage{color}
    \usepackage{fancyhdr}
    \usepackage{rotating}
    \usepackage{fancybox}
    \usepackage{colortbl}
    \usepackage{pifont}
    \usepackage{setspace}
    \usepackage{enumerate}
    \usepackage{multicol}
    \usepackage{varioref}
    \usepackage{lmodern}
    \usepackage{textcomp}
    \usepackage{euscript}
    \usepackage[pdftex]{hyperref}
			\newtheorem{theorem}{Theorem}[section]

			\newtheorem{proposition}[theorem]{Proposition}

			\newtheorem{remark}[theorem]{Remark}

\newcommand{\norm}[1]{\left\Vert#1\right\Vert}  \newcommand{\scal}[1]{\left<#1\right>}
\newcommand{\R}{\mathbb{R}}    \newcommand{\C}{\mathbb{C}}

     \newcommand{\bz}{\overline{z}} \newcommand{\bw}{\overline{w}}
    \newcommand{\bxi}{\overline{\xi}}
    \newcommand{\bzeta}{\overline{\zeta}}  
     \newcommand{\magn}{\nu} 
\begin{document}
%\begin{frontmatter}
%%%%%%%%%%%%%%%%%%%%%%%%%%%%%%%%%%%%%%%%%%%%%%%%%%%%%%%%%%%%%%%%%%%%%%%%%%%%%%%%%%%%%%%%%%%%%%%%%%%%%%%%%%%%%%%%%%%%%%%%%%%%%%%%%%%%
%\title[A class of integral transforms with range in  weighted Bergman space on the bi-disk]{A class of integral transforms associated yo It\^o--Hermite polynomials with range in  weighted Bergman space on the bi-disk}

\title[Dual of $2$D FrFT associated to It\^o--Hermite polynomials]{Dual of $2$D fractional Fourier transform associated to It\^o--Hermite polynomials %with range in  weighted Bergman space on the bi-disk
}

 \dedicatory{ \textit{Dedicated to the memory of Professor Elhachmia Ait Benhaddou}}

\author{Abdelhadi Benahmadi} \email{abdelhadi.benahmadi@gmail.com}
\author{Allal Ghanmi}  \email{allalghanmi@um5.ac.ma}
%\dedicatory{Dedicated to Professor Ahmed Intissar on the occasion of his 65th birthday }
\address{ Analysis, P.D.E. $\&$ Spectral Geometry,
	Lab M.I.A.-S.I., CeReMAR, Department of Mathematics,
	P.O. Box 1014,  Faculty of Sciences,\newline
	Mohammed V University in Rabat,
	Morocco}

      
      %Mathematics Subject Classification (2020)
      %	44A20   MSC2020: Integral transforms of special functions
      %	30G35   MSC2020: Functions of hypercomplex variables and generalized variables
      %	30H20  MSC2020: Bergman spaces and Fock spaces	
      %	Secondary 
      %	47B38  MSC2020: Linear operators on function spaces (general)
      
      
      \date{}

\begin{abstract}
A class of integral transforms, on the planar Gaussian Hilbert space with range in the weighted Bergman space on the bi-disk, is defined as the dual transforms of the $2$d fractional Fourier transform associated with the Mehler function for It\^o--Hermite polynomials. Some spectral properties of these transforms are investigated. Namely, we study their boundedness and identify their null spaces as well as their ranges. Such identification depends on the zeros set of It\^o--Hermite polynomials. Moreover, the explicit expressions of their singular values are given and compactness and membership in $p$-Schatten class are studied. The relationship to specific fractional Hankel transforms is also established  
\end{abstract}

\maketitle

%%%%%%%%%%%%%%%%%%%%%%%%%%%%%%%%%%%%%%%%%%%%%%%%%%%%%%%%%%%%%%%%%%%%%%%%%%%%%%%%%%%%%%%%%%%%%%%%%%%%%%%%%%%%%%%%%%
\section{Introduction } \label{s1}

The role played by the classical Mehler formula \cite{Mehler1866},
\begin{align}\label{MehlerkernelHn}
\sum\limits_{n=0}^{+\infty} \frac{ t^n H_{n}(x) H_{n}(y)}{2^n n!}
=\frac{1}{\sqrt{1 - t^2}}  \exp\left( \frac{- t^2 (x^2 + y^2) + 2 t x y  }{1 - t^2}  \right) ,
\end{align}
for the real Hermite polynomials $H_{n}(x) : = (-1)^n e^{x^2} \frac{d^n}{dx^n}(e^{-x^2}),$
%in different areas of harmonic analysis, Fourier analysis, probability, quantum physics, and engineer sciences, 
is well known in the literature \cite{Mehler1866,Wiener 1929,Condon1937,Kibble1945,Slepian1972,SrivastavaSinghal1972,Louck1981,Hormander1995,Stanton2000}. Its complex analogues for It\^o--Hermite polynomials $H_{m,n}^\nu$ have been obtained in \cite{Wunsche1999,ChenLiu2014,IsmailTrans2016,Gh2017Mehler} and have been employed in \cite{Gh2017Mehler} to establish integral reproducing property for $H^\nu_{m,n}$ by a like Fourier transform,
%Indeed, we have 
%	\begin{align}\label{eigenFourier1} %
%	\int_\C    e^{ i \Re(zw)} e^{-\frac{|w|^2}{2}} H_{k,j}(w;\bw) d\lambda(w) = 2 \pi   i^{j+k}  e^{-\frac{|z|^2}{2}} H_{j,k}(z;\bz)
%	\end{align}
%	This is the analogue of the real Hermite functions being eigenfunctions of the Fourier transform  \cite{Andrews}.
%\end{remark}
and to provide a closed expression of the heat kernel for the Cauchy initial value problem
%$$ (H) \quad \left\{ \begin{array}{ll} \frac{\partial }{\partial t} u(t;z) = \Delta_{\magn} u(t;z); & (t;z)\in ]0,+\infty[\times \C, \\ u(t;z) = f(z) \in \mathcal{C}^\infty_0(\C), \end{array} \right.%; \, z\in \C,
%$$
for a special magnetic Laplacian 
%\begin{align}\label{MagnLap}
%\Delta_{\magn}  =  - \dfrac{\partial^2}{\partial z\partial\bz  } + \magn z  \dfrac{\partial}{\partial z }
%\end{align}
acting on the Hilbert space $L^{2}_\nu(\C) := L^{2,\magn }(\C; e^{-\magn |z|^2}d\lambda)$. In \cite{Zayed2018}, Zayed has used the one in  \eqref{Mehler20closed} below to construct a non trivial $2$d fractional Fourier transform  
 \begin{align}\label{2dFrFTZ}
 \mathcal{F}^{\nu}_{u,v}\psi (\xi)
 = \int_{\C} \psi(\zeta) K^\nu_{u,v}( \zeta ; \xi) e^{-\nu |\zeta|^2} dxdy; \,\, \zeta=x+iy  ,
 \end{align}
 whose eigenfunctions are the It\^o--Hermite polynomials.
 Here $K^\nu_{u,v}( \zeta ; \xi )$ stands for the kernel function  
 \begin{align} \label{2dFrFTKernel2}
 K^\nu_{u,v}( \zeta ; \xi )
 &= \frac{\magn}{\pi(1-uv)}  \exp\left(\frac{\nu}{1-uv}  \left\{- uv (|\zeta|^2+|\xi|^2)  +  u \bzeta\xi +  v \zeta \bxi\right\} \right).
 \end{align}

 In the present paper, we explore further applications in the context of integral transforms and weighted Bergman spaces $ B^2_{\alpha,\beta} (D^2)$ on the bi-disk $D^2=D\times D$; $D=\{z\in \C, z\bz <1\}$, defined 
 as the Hilbert space of all analytic functions on $D^2$ that are square integrable with respect to the measure 
  \begin{align}
d\mu_{\alpha,\beta} (z,w)=  \omega_{\alpha,\beta}(|z|^2,|w|^2) d\lambda(z,w) ,
 \end{align}
 where the weight function is given by 
   \begin{align}
  \omega_{\alpha,\beta}(s,t) := (1-s)^{\alpha}(1-t)^{\beta } ; \, \alpha, \beta >-1,
 \end{align}
 and $d\lambda$ denotes the standard Lebesgue measure. To this end we follow the scheme already applied in \cite{GhDual2020} to introduce and study the dual transforms of fractional Hankel transforms with ranges in weighted Bergman space on the disk.
 Mainly, we consider the family of integral transforms  
 \begin{align}\label{TransR}
 \mathcal{R}_w^{\nu}f(u,v)=\int_{\C}f(z)K^\nu_{u,v}( z ; w )e^{-\nu |z|^2 }d\lambda(z) 
 \end{align}
 on $D^2$, labeled by $\nu>0$ and $w\in \C$ and seen as the dual transform of the $2$d fractional Fourier transform in \eqref{2dFrFTZ}, $ \mathcal{R}_w^{\nu}f(u,v) =  \mathcal{F}^{\nu}_{u,v}f (w)$.

The aim in this paper concern identification of the null space and the range of the transforms $\mathcal{R}_w^\nu$. We also study their boundedness and provide complete description of their compactness and membership in $p$-Schatten class. Our main results can be stated as follows

\begin{theorem}\label{MThm1}
	The integral transform $\mathcal{R}_w^{\nu}$ is well defined and bounded from $L^2_{\nu}(\C)$ into the weighted Bergman space $B^2_{\alpha,\beta} (D^2)$ if and only if $\alpha>0$ and $\beta>0$. The characterization of its null space $Ker(\mathcal{R}_w^{\nu})$ depends on the zeros set of It\^o--Hermite polynomials $H^{\nu}_{m,n}$. Namely, if $ N_{w}(H) = \{(m,n); m,n=0,1,2, \cdots ; \, H^\nu_{m,n}(w,\bw)=0\} $, then $Ker(\mathcal{R}_w^{\nu})$ is a vector space spanned as
	$$Ker(\mathcal{R}_w^{\nu}) =  Span\{H^\nu_{m,n}; \, (m,n)\in N_{w}(H) \}.$$
\end{theorem} 

\begin{theorem}\label{MThm2}
	For $\alpha,\beta>0$, the operator $\mathcal{R}_w^{\nu}: L^2_{\nu}(\C) \longrightarrow B^2_{\alpha,\beta} (D^2)$ is compact and its singular values are given by 
	$$s^{\nu,\alpha,\beta}_{m,n}(w)=
	\left(  \frac{\nu \pi \Gamma(\alpha+1) \Gamma(\beta+1) }{\nu^{m+n} \Gamma(\alpha+m+2)\Gamma(\beta+n+2)} \right)^{1/2} |H^{\nu}_{m,n}(w,\bw)|.$$
	Moreover, it belongs to the $p$-Schatten class for every $p>\max(2/(\alpha+1); 2/(\beta+1))$.
\end{theorem}

The proof of Theorem \ref{MThm1} is contained in Propositions \ref{propwelldef}, \ref{propBound}, \ref{proprange} and \ref{propnull} presented in Section 3, while the one of Theorem \ref{MThm2} is given in Section 4. The next section is devoted to some preliminaries concerning weighted Bergman space on the bi-disk and It\^o--Hermite polynomials. We conclude the paper by discussing the close connection of $ \mathcal{R}_w^{\nu}$ to the fractional Hankel transforms 
\begin{align}\label{GFrHT}
\mathcal{H}_{u,v}^{\nu,\alpha}(f)(y)&=
	\frac{2\nu}{1-uv} \left( \frac{u}{v}\right)^{\alpha/2}   \int_0^\infty x f(x) I_{\alpha} \left( \frac{2\nu \sqrt{uv }}{1-uv}  xy \right)  e^{\frac{-\nu (x^2  +uv y^2)}{1-uv} } dr,
\end{align}
where $I_\alpha$ denotes the modified Bessel function \cite[p.222]{AndrewsAskeyRoy1999}. 
%		\begin{align}\label{BesselFct} I_\alpha (\xi) :=    \sum_{n=0}^{\infty} \frac{1}{n! \Gamma(\alpha+n+1)} \left(  \frac{\xi}{2}\right)^{2n+\alpha}.\end{align}
%It is clearly a specific generalization of the Hankel transform \cite{}
%$$ \mathcal{H}_{m}(\rho) = \int_0^\infty f(r) J_m(\rho r) \,r\,\mathrm{d}r.$$

\begin{theorem}\label{MThm3}
%For every rotational symmetric function  $\psi(\zeta) = \Psi(|\zeta|^2) e^{ik\theta}$,, we have 
%\begin{align*}
%\mathcal{F}^{\nu}_{u,v}\psi (\xi)
%&=   
%\frac{2\nu (-i)^{k} }{1-uv} \left( \frac{ v\bxi}{  u\xi}\right) ^{k/2}  
%\int_0^\infty r\psi(re^{i\theta}) e^{-ik\theta} J_{k} \left( \frac{2i\nu \sqrt{uv }} {1-uv}|\xi| r \right)  e^{\frac{-\nu (r^2+uv|\xi|^2)}{1-uv}  } dr .
%\end{align*}
%The fractional Fourier coefficients $G_k$ and $g_k$ of given  $f\in L^2_{\nu}(\C)$ and its fractional Fourier transform $\mathcal{F}^{\nu}_{u,v}f$, respectively, are connected by 
%$ G_k  = \mathcal{H}_{u,v}^{\nu,\alpha}(g_k).$
Let $f\in L^2_{\nu}(\C)$ and $g_k$ the associated Fourier coefficients. Then, $\mathcal{H}_{u,v}^{\nu,\alpha}(g_k)$, for varying integer $k$, are the fractional Fourier coefficients of $\mathcal{F}^{\nu}_{u,v}f$,  the  fractional Fourier transform of $f$.
%\begin{align*}
%G_m(\rho) =  	\frac{2\nu (-i)^{m} }{1-uv} \left( \frac{ v}{u}\right)^{m/2}   \int_0^\infty r g_{m}(r) J_{m} \left( \frac{2i\nu \sqrt{uv }}{1-uv}  \rho r \right)  e^{\frac{-\nu (r^2  +uv \rho^2)}{1-uv} } dr 
%\end{align*}
\end{theorem} 
%\begin{theorem}\label{MThm3}
%		The fractional Fourier coefficients $G_m$ and $g_m$ of given  $\psi \in L^2_{\nu}(\C)$ and its fractional Fourier transform $\mathcal{F}^{\nu}_{u,v}\psi$, respectively, are connected by 
%	$$ G_m  = \mathcal{H}_{u,v,m}^{\nu}(g_m).$$
%\end{theorem} 

\section{Preliminaries}

For fixed reals $\alpha,\beta >-1$, the weighted Bergman space $ B^2_{\alpha,\beta} (D^2)$ is a closed subspace of the Hilbert space $L^2_{\alpha,\beta} (D^2):=L^2\left(\mathbb{D}\times \mathbb{D};d\mu_{\alpha,\beta}\right) $ endowed with the scaler product  
$$ \scal{f,g}_{\alpha,\beta} = \int_{D^2} f(z,w) \overline{g(z,w)} d\mu_{\alpha,\beta}(z,w).$$
We denote by $\norm{\cdot}_{\alpha,\beta}$ the associated norm.
An orthonormal basis of $B^2_{\alpha,\beta} (D^2)$ is given by 
$$ 
\varphi^{\alpha,\beta}_{m,n}:= \left(\gamma^{\alpha,\beta}_{m,n} \right)^{-1/2} e_{m,n}
%\left( \frac{\Gamma(\alpha+m+2)\Gamma(\beta+n+2)}{\pi^2 \Gamma(\alpha+1) \Gamma(\beta+1) m!n!}\right)^{1/2} e_{m,n},
$$ 
where $e_{m,n}(z,w):=z^mw^n$ and $\gamma^{\alpha,\beta}_{m,n}$ is its square norm  given by 
%Indeed, we have 
%\begin{align*}
%\scal{e_{m,n},  e_{j,k}}_{\alpha,\beta}  %M&= \left( \int_{D} u^{m}\overline{u}^j (1-|u|^2)^{\alpha} d\lambda(u) \right)   \left( \int_{D} v^{n}\overline{v}^k (1-|v|^2)^{\beta} d\lambda(v) \right) \\
%&	  = 
%\pi^2  
%\left( \int_{0}^1 t^{m} (1-t)^{\alpha} dt \right)   \left( \int_{0}^1 t^{n}  (1-t)^{\beta} dt \right) \delta_{m,j}\delta_{n,k}	 
%\\&	 =m!n! \frac{\Gamma(\alpha+1) \Gamma(\beta+1)}{\Gamma(\alpha+m+2)\Gamma(\beta+n+2)} \delta_{m,j}\delta_{n,k}.
%\end{align*} 
\begin{align*}
\gamma^{\alpha,\beta}_{m,n} 
:=  \frac{\pi^2 \Gamma(\alpha+1) \Gamma(\beta+1)m!n!}{\Gamma(\alpha+m+2)\Gamma(\beta+n+2)}=  \norm{e_{m,n} }^2_{\alpha,\beta} .
\end{align*}
Subsequently, the sequential characterization of $ B^2_{\alpha,\beta} (D^2)$ is given by 
$$ B^2_{\alpha,\beta} (D^2) = \left\{ %f(z,w) = 
\sum_{m,n=0}^\infty a_{m;n} e_{m,n} ; \,  
\sum_{m,n=0}^\infty 
\gamma^{\alpha,\beta}_{m,n} |a_{m;n} |^2 < \infty
 \right\} .$$
Accordingly, $ B^2_{0,0} (D^2)$
%$$ B^2_{0,0} (D^2) = \left\{ f(z,w) = \sum_{m,n=0}^\infty a_{m;n} z^mw^n; \,  \pi^2 \sum_{m,n=0}^\infty \frac{|a_{m;n} |^2}{(m+1)(n+1)}  < \infty \right\} $$
 is identified to be the Hardy space on the bi-disk, while $ B^2_{-1,-1} (D^2)$
% $$ B^2_{-1,-1} (D^2) = \left\{ f(z,w) = \sum_{m,n=0}^\infty a_{m;n} z^mw^n; \,  \pi^2 \sum_{m,n=0}^\infty  |a_{m;n} |^2 < \infty \right\} $$
  is the classical Bergman space on $D^2$.
The reproducing kernel of $ B^2_{\alpha,\beta} (D^2)$ is given by 
\begin{align*}
K_{\alpha,\beta}((u,v);(z,w)) 
%=  \frac{(\alpha+1)(\beta+1)}{\pi^2} {_1F_0}\left(   \begin{array}{c}   \alpha+1 \\ - \end{array} \bigg | u\overline{z} \right) {_1F_0}\left(   \begin{array}{c}   \beta+1 \\ - \end{array} \bigg | v\overline{w} \right) 
=  \frac{(\alpha+1)(\beta+1)}{\pi^2( 1- u\overline{z})^{\alpha+2}
( 1-v\overline{w})^{\beta+2}}  .
\end{align*}
Thus $ B^2_{\alpha,\beta} (D^2)$ is obtained as the Bergman  projection 
$$P(\varphi)(u,v) 
%=  \scal{ \varphi ,\overline{K_{\alpha,\beta}((u,v);\cdot )} }_{\alpha,\beta} 
=\frac{(\alpha+1)(\beta+1)}{\pi^2} \int_{D^2}  \frac{\varphi (z,w)}{( 1- u\overline{z})^{\alpha+2}
 	( 1-v\overline{w})^{\beta+2}} d\mu_{\alpha,\beta}(z,w) $$ 
of $L^2_{\alpha,\beta}(D^2)$. 
Another realization is by means of the unitary two-dimensional second Bargmann transform 
	\begin{equation*}%\label{2dSecondBargmann}
\mathcal{B}_{\alpha,\beta} \varphi(z,w)=   
\frac{1}{\left( 1-z\right)^{\alpha+1}\left( 1-w\right)^{\beta+1} } \int_{\R^{+2}}  s^\alpha t^\beta \exp\left(\frac{ sw+tz - ( s + t ) }{(1-z)(1-w)} \right)  \varphi(s,t)   ds dt
\end{equation*}
acting on the Hilbert space $L^{2}(\R^{+2}; x^{\alpha}y^{\beta}e^{-x-y}dxdy)$. The  kernel function of $\mathcal{B}_{\alpha,\beta}$ appears as the tensor product of two copies of the kernel function  of the standard one-dimensional second Bargmann transform \cite[p. 203]{Bargmann1961}. 
However, it can be seen as the bilinear generating function involving the product of generalized Laguerre polynomials $ L^{(\alpha)}_m(s) L^{(\beta)}_n(t)$.
%In fact we have 
%\begin{align*}
%R_{\alpha,\beta}(x,y;z,w)&:= \sum_{m,n=0}^\infty \varphi^\alpha_{m}(x)\varphi^\beta_{n}(y) \frac{e_{m,n}(z,w)}{\norm{e_{m,n}}} 
%\\&=\left( \frac{\Gamma(\sigma +m+1)\Gamma(\sigma' +n+1)}{\pi^2m!n!\Gamma(\sigma+1)\Gamma(\sigma'+1)}\right)^{1/2}
%\left( \frac{m!n!}{\Gamma(\alpha +m+1)\Gamma(\alpha +n+1)}\right)^{1/2}
%z^mw^n L^{(\alpha)}_{m}(x)L^{(\alpha)}_{n}(y)
%\\&=\frac{1}{\pi\sqrt{\Gamma(\sigma+1)\Gamma(\sigma'+1)}} \left( \frac{\Gamma(\sigma +m+1)\Gamma(\sigma' +n+1)}{\Gamma(\alpha +m+1)\Gamma(\alpha +n+1)} \right)^{1/2}
%z^mw^n L^{(\alpha)}_{m}(x)L^{(\alpha)}_{n}(y)
%\\&=\sum_{m,n=0}^\infty z^mw^n  L^{\alpha}_m(x) L^{\beta}_n(y)
%\\&= \frac{1}{(1-z)^{\alpha+1}} e^{\frac{-xz}{1-z}} \frac{1}{1-w)^{\beta+1}} e^{\frac{-yw}{1-w}}.
%\end{align*}

In the sequel, we will provide interesting realization of specific subspaces of $ B^2_{\alpha,\beta} (D^2)$ by invoking the complex Mehler function  
\begin{align}\label{Mehler20closed}
K^{\nu;\nu'}_{u,v}(z,w)  
=\frac{1}{1 -  uv}  \exp\left( \frac{ -uv(\nu |z|^2 + \nu'|w|^2)  + \nu' uzw + \nu v\overline{z}\overline{w}  }{1 -   uv} \right),
\end{align}
with $\nu,\nu'>0 $ and $u,v \in D$, associated to It\^o--Hermite polynomials $H_{m,n}^\nu $ defined on the complex plane $\C$ by
\cite{Ito52,IntInt06,Gh13ITSF}   
\begin{align}\label{gchpmu}
H_{m,n}^\nu  (z,\bz )=(-1)^{m+n}e^{\magn  z\bz }\dfrac{\partial ^{m+n}}{\partial \bz^{m} \partial z^{n}} \left(e^{-\magn  z \bz }\right) .
\end{align}
%whose representation in terms of the confluent hypergeometric function ${_1F_1}$ is given by
%%	\begin{align}\label{HypergeometricRep}
%%H_{m,n}(z,\bz)=\dfrac{(-1)^{\min(m,n)}\max(m,n)!}{|m-n|!} z^m\bz^n |z|^{|m-n|-m-n}  {_1F_1}\left( \begin{array}{c} -\min(m,n) \\ |m-n|+1 \end{array}\bigg | |z|^{2} \right).
%%\end{align}
%	\begin{align}\label{HypergeometricRep}
%H^\nu_{m,n}(z,\bz)=\left\{\begin{array}{lll} 
%\dfrac{(-1)^{m}n!}{(n-m)!} \bz^{n-m}  {_1F_1}\left( \begin{array}{c} -m\\ n-m+1 \end{array}\bigg | \nu|z|^{2} \right); \quad m\leq n\\
%\\
%\dfrac{(-1)^{n} m!}{(m-n)!} z^{m-n}  {_1F_1}\left( \begin{array}{c} -n \\ m-n+1 \end{array}\bigg | \nu|z|^{2} \right) ; \quad m\geq n
%\end{array} \right.
%\end{align}
To this end, let recall that the $H_{m,n}^\nu $  form an orthogonal basis of $L^2_\nu(\C)$ and that the Mehler function in \eqref{Mehler20closed}
 can be expanded in terms of normalized It\^o--Hermite polynomials 
 	\begin{align}\label{basisL2c}
 \psi^\nu_{m,n}:= \left( \frac{\nu}{\pi \nu^{m+n} m!n!} \right)^{1/2} H^{\nu'}_{m,n}
 \end{align}
 as \cite{Gh2017Mehler} 
%\begin{align}\label{Mehler20}
% K^{\nu;\nu'}_{u,v}(z,w) :=\sum_{m,n=0}^{\infty}  \left( \frac{u}{\nu}\right)^m 
%\left( \frac{v}{\nu'}\right)^n \frac{H_{m,n}^{\nu }(z;\bz) H_{m,n}^{\nu'}(w;\bw)}{m!n! } 
%\end{align}
\begin{align}\label{Mehler20}
K^{\nu;\nu'}_{u,v}(z,w) :=\sum_{m,n=0}^{\infty}    u^m 
v^n \psi^\nu_{m,n}(z) \psi^\nu_{m,n}(w) .
\end{align}
For $\nu=\nu'=1$, this is exactly the one announced by W\"unsche \cite{Wunsche1999} and proved later by Ismail in \cite[Theorem 3.3]{IsmailTrans2016} as a specific case of his Kibble--Slepian formula \cite[Theorem 1.1]{IsmailTrans2016}. 

\section{Basic properties of $\mathcal{R}_w^{\nu}$}

We begin by observing that the kernel function in \eqref{2dFrFTKernel2} reads in terms of the one in \eqref{Mehler20closed} as
% \begin{align} 
%K^\nu_{u,v}( \bzeta ; \xi )&= \frac{\magn}{\pi(1-uv)}  \exp\left(\frac{\nu\left\{- uv (|\zeta|^2+|\xi|^2)  +  u \zeta\xi +  v \bzeta \bxi\right\}}{1-uv}   \right)=\left( \frac{\nu}{\pi}\right) K^{\nu;\nu}_{u,v}(\zeta,\xi).
%\end{align} 
\begin{align}\label{Mehler20kk}
K^{\nu}_{u,v}(z,w)  = \left( \frac{\nu}{\pi}\right) K^{\nu;\nu}_{u,v}(\bz,w) .
\end{align} 
and therefore satisfies 
\begin{align}\label{kerker}
\int_{\C} K^\nu_{u,v}( z; w) \overline{K^\nu_{u,v}( z; w)}  e^{-\nu|z|^2} d\lambda(z)  =K^\nu_{|u|^2,|v|^2}(w;w) >0. 
\end{align}
%This is to say that the Mehler formula \eqref{Mehler20closed}  can be used to describe the basic properties of the transform $\mathcal{R}_w^{\nu}$.

\begin{proposition} \label{propwelldef}
		The integral transform $\mathcal{R}_w^{\nu}$ is well defined on $L^2_{\nu}(\C)$. 
\end{proposition}

\begin{proof}
Using \eqref{kerker}
%%and
%%\textcolor{red}{   
%%\begin{align}\label{estimate2}
%%0<K^{\nu}_{x,y}(z,z)  
%%%= \left( \frac{\nu}{\pi}\right)  \sum_{m,n=0}^{\infty}   \frac{x^my^n}{\nu^{m+n}}  \frac{|H_{m,n}^{\nu }(z;\bz)|^2 }{m!n! }  
%%\leq \left( \frac{\nu}{\pi}\right)  \sum_{m,n=0}^{\infty}    \frac{ x^m }{m!n!\nu^{m+n} }  |H_{m,n}^{\nu }(z;\bz)|^2 
%% \leq  \left( \frac{\nu}{\pi(1-x)}\right)   e^{ \nu |z|^2  } 
%%\end{align}
%%}
%% for every reals $0\leq x,y <1$, by means of \cite[Corollary 3.2]{Gh13ITSF}.
and the Cauchy-Schwarz inequality, 
%combined with \eqref{kerker},
 we obtain 
\begin{align}
|\mathcal{R}_w^{\nu}f(u,v)| 
%&= |\left\langle f,K^\nu_{u,v}(  ; w )\right\rangle_{L^2_{\nu}(\C)}|
&\leq \left( \int_{\C} |K^\nu_{u,v}( w; w)|^2 \right) ^{1/2} \norm{f}_{L^2_{\nu}(\C)} 
 \nonumber
\\&\leq \left( K^{\nu}_{|u|^2,|v|^2}(w; w )\right)^{1/2} \norm{f}_{L^2_{\nu}(\C)} \label{estimateCS1}
\end{align}
for every $f\in L^2_{\nu}(\C)$. 
\end{proof}

%Another proof
%\begin{proof}
%	Using the Cauchy-Schwarz inequality we obtain 
%	\begin{align*}
%	|\left\langle f,K^\nu_{u,v}(  ; w )\right\rangle_{L^2_{\nu}(\C)}|&\leq \norm{K^\nu_{u,v}(  ; w )}_{L^2_{\nu}(\C)} \norm{f}_{L^2_{\nu}(\C)} \\
%	&= \left( \sum_{m,n=0}^\infty   \frac{|u|^{2m} |v|^{2n}}{m!n!\nu^{m+n}} |H^\nu_{m,n}(w,\overline{w})|^2\right)^{1/2} \norm{f}_{L^2_{\nu}(\C)}\\
%	&\leq \frac{\nu}{\pi} \left( \sum_{m,n=0}^\infty |u|^m |v|^n   e^{\nu|w|^2} \right)^{1/2} \norm{f}_{L^2_{\nu}(\C)}
%	\\ &\leq \frac{\nu}{\pi} \left( \frac{1}{1- |uv|^2}  e^{\nu|w|^2} \right)^{1/2} \norm{f}_{L^2_{\nu}(\C)}
%	\end{align*}
%	for every $f\in L^2_{\nu}(\C)$. %, whenever $K^\nu_{u,v}(  ; w )\in L^2_{\nu}(\C)$.
%	The mast inequality is due to $|H^\nu_{m,n}(w,\overline{w})|^2 \leq \frac{\nu}{\pi}  \norm{H^{\nu}_{m,n}}^2 e^{\nu|w|^2}$ which can follows from the generating function 
%	\begin{eqnarray} \label{prob1}
%	\sum\limits_{n=0}^{+\infty} \frac{|H_{m,n}^{\magn }(z,\bz )|^2  }{n! \magn ^n  }   &=   m! \magn ^{m}  e^{\magn |z|^2 }.
%	\end{eqnarray}
%\end{proof}

 The action of $\mathcal{R}_w^{\nu}$ on $\psi^\nu_{m,n}$ in \eqref{basisL2c} is given by 
	\begin{align*}
	\mathcal{R}_w^{\nu}\psi^\nu_{m,n}	= \psi^\nu_{m,n}(w) e_{m,n}. 
	\end{align*}
This follows by means of \eqref{Mehler20kk} and \eqref{Mehler20}.
Therefore, the family $\mathcal{R}_w^{\nu}\psi^\nu_{m,n}$, for varying $m,n$, form an orthogonal system in $L^2_{\alpha,\beta} (D^2)$ since the monomials $e_{m,n}$ are. 
%Moreover, we have 
%\begin{align*}
%\norm{\mathcal{R}_w^{\nu}H_{m,n}}^2_{\alpha,\beta}
%&=|H^\nu_{m,n}(w,\overline{w})|^2 \norm{e_{m,n}}^2_{\alpha,\beta}	 
%=m!n!|H^\nu_{m,n}(w,\overline{w})|^2\frac{\Gamma(\alpha+1) \Gamma(\beta+1)}{\Gamma(\alpha+m+2)\Gamma(\beta+n+2)} .
%\end{align*}  
The next result discusses the boundedness of $\mathcal{R}_w^{\nu}$ from $L^2_{\nu}(\C)$ into the weighted Hilbert space 
$L^2_{\alpha,\beta} (D^2)$.

\begin{proposition}\label{propBound}
	For $\alpha>0$ and $\beta>0$, the operator $\mathcal{R}_w^{\nu}$ is  bounded from $L^2_{\nu}(\C)$ into $L^2_{\alpha,\beta} (D^2)$.
\end{proposition}

\begin{proof} Set	\begin{align}
k_w^{\nu,\alpha,\beta} :=	\int_{D^2}K^\nu_{|u|^2,|v|^2}(w; w ) d\mu_{\alpha,\beta}(u,v)  .
	\end{align}
Then, from \eqref{estimateCS1}, we have 		
	\begin{align}
	%k_w^{\nu,\alpha,\beta} \leq	
	%\sup_{f\in L^2_{\nu}(\C)}  
	\norm{\mathcal{R}_w^{\nu}f}^2_{\alpha,\beta} \leq 
	k_w^{\nu,\alpha,\beta} 
	 %\sup_{f\in L^2_{\nu}(\C)}  
	\norm{f}^2_{L^2_{\nu}(\C)} .  
	\end{align}
Subsequently, the boundedness of the operator $\mathcal{R}_w^{\nu}$ requires that $k_w^{\nu,\alpha,\beta}$ be finite.
%This is equivalent to the convergence of the series 
%\begin{proof} Set	\begin{align}
%	k_w^{\nu,\alpha,\beta}
%	&= \nu \pi  \sum_{m,n=0}^\infty   \frac{\Gamma(\alpha+1) \Gamma(\beta+1)} {\nu^{m+n} \Gamma(\alpha+m+2) \Gamma(\beta+n+2)} \left| H_{m,n}^{\nu}(w;\bw)\right|^2.
%	\end{align}	
But, using the closed expression of $K^\nu_{|u|^2,|v|^2}(w; w )$, we get
\begin{align}
k_w^{\nu,\alpha,\beta} 
%	= \left( \frac{\nu}{\pi}\right)  \int_{D^2}\frac{1}{1 -  |uv|^2}  \exp\left(  \frac{ \nu(|u|^2+|v|^2-2|uv|^2) }{1 -   |uv|^2} |w|^2 \right) d\mu_{\alpha,\beta}(u,v) 
= \nu\pi    \int_0^1\int_0^1   \exp\left(  \frac{\nu(s+t-2st)}{1-st} |w|^2 \right) \frac{\omega_{\alpha,\beta}(s,t)}{1-st} dsdt.
\end{align}
%
%\textcolor{red}{
%	\begin{align}
%	k_w^{\nu,\alpha,\beta} 
%	= \left( \frac{\nu}{\pi}\right)  \int_{D^2}\frac{1}{1 -  |uv|^2}  \exp\left(  \frac{ \nu(|u|^2+|v|^2-2|uv|^2) }{1 -   |uv|^2} |w|^2 \right) d\mu_{\alpha,\beta}(u,v) 
%	\end{align}
%}
%Since $s,t\in(0,1)$ we have $s(1-t)\leq 1-t$ and  $s+t -st \leq 1$. Hence $0\leq s+t -st -st \leq 1-st$
%\begin{align}
%\nu\pi    \int_0^1\int_0^1  (1-s)^\alpha (1-t)^\beta dsdt
%\leq	k_w^{\nu,\alpha,\beta} 
%\leq \nu\pi e^{\nu |w|^2}   \int_0^1\int_0^1  (1-s)^{\alpha-1} (1-t)^{\beta-1} dsdt.
%\end{align}
Hence since $0\leq (s+t -2st)/(1-st) \leq 1$ and $1/(1-st) \leq 1/(1-s)(1-t)$, it follows
\begin{align}
\nu\pi    \int_0^1\int_0^1 \omega_{\alpha,\beta}(s,t)dsdt
\leq	k_w^{\nu,\alpha,\beta} 
\leq \nu\pi e^{\nu |w|^2}   \int_0^1\int_0^1  \omega_{\alpha-1,\beta-1}(s,t)  dsdt.
	\end{align}
%which requires $\alpha,\beta>0$.
Thus, $\mathcal{R}_w^{\nu}$ is bounded for $\alpha,\beta>0$. In this case $\mathcal{R}_w^{\nu}f$ belongs to $L^2_{\alpha,\beta} (D^2)$ for every $f\in L^2_{\nu}(\C)$.
\end{proof}
 
Now,  appealing to the fact that $\psi^\nu_{m,n}$ constitutes an orthonormal basis of $L^2_{\nu}(\C)$, we can expand any $f\in L^2_{\nu}(\C)$ as $f=\sum\limits_{m,n=0}^\infty\alpha_{m,n} \psi^\nu_{m,n}$,
% for some complex sequence such that $\sum\limits_{m,n=0}^\infty m!n!|\alpha_{m,n}|^2 <+\infty$, 
so that one gets 
\begin{align}\label{ExprRf}
%\mathcal{R}_w^{\nu}f = \sum\limits_{m,n=0}^\infty\alpha_{m,n}
%\psi^\nu_{m,n}(w) e_{m,n} 
\mathcal{R}_w^{\nu}f = \sum\limits_{m,n=0}^\infty\alpha_{m,n}
\psi^\nu_{m,n}(w)  \left(\gamma^{\alpha,\beta}_{m,n} \right)^{1/2} \varphi^{\alpha,\beta}_{m,n} .
%&=   \sum\limits_{m,n=0}^\infty   |\alpha_{m,n} |^2 |\psi^\nu_{m,n}(w) |^2 \gamma^{\alpha,\beta}_{m,n} 
\end{align}
The series in \eqref{ExprRf} converges uniformly on compact sets of the complex plane.
Direct computation shows that we have 
\begin{align}\label{normRf}
\norm{\mathcal{R}_w^{\nu}f}^2 _{\alpha,\beta}
%&= \int_{D^2}	|\mathcal{R}_w^{\nu}f(u,v)|^2 d\mu_{\alpha,\beta }(u,v) \\&=	\sum\limits_{m,n=0}^\infty |\alpha_{m,n} |^2	\norm{\mathcal{R}_w^{\nu} H^\nu_{m,n}}^2_{\alpha,\beta} 
%\\& =\sum\limits_{m,n=0}^\infty |\alpha_{m,n} |^2|H^\nu_{m,n}(w,\overline{w})|^2\norm{e_{m,n}}^2_{\alpha,\beta} \\
%&=\pi^2\Gamma(\alpha+1) \Gamma(\beta+1)\sum\limits_{m,n=0}^\infty \frac{m!n!|H^\nu_{m,n}(w,\overline{w})|^2}{\Gamma(\alpha+m+2)\Gamma(\beta+n+2)} |\alpha_{m,n} |^2
&=   \sum\limits_{m,n=0}^\infty   |\alpha_{m,n} |^2 |\psi^\nu_{m,n}(w) |^2 \gamma^{\alpha,\beta}_{m,n} .
\end{align}

Accordingly, the description of the range and the null space of the $\mathcal{R}_w^{\nu}$  are closely connected to zeros of It\^o--Hermite polynomials. Thus, we let $\mathcal{Z}(H^\nu_{m,n})$ denotes the zeros set of $H^\nu_{m,n}$ for fixed $m,n$, while $\mathcal{Z}(H) := \cup_{m,n} \mathcal{Z}(H^\nu_{m,n})$. We also set 
$$ N_{w}(H) = \{(m,n); m,n=0,1,2, \cdots ; \, H^\nu_{m,n}(w,\bw)=0\} .$$

\begin{proposition}\label{proprange} Let $\alpha>0$ and $\beta>0$. If $w \notin \mathcal{Z}(H)$, then the range of $\mathcal{R}_w^{\nu}$ acting on $L^2_{\nu}(\C)$ is a specific subspace of the weighted Bergman space $B^2_{\alpha,\beta} (D^2)$.   
\end{proposition}

\begin{proof}
This is immediate by means of \eqref{ExprRf} and Proposition \ref{propBound}. Indeed, $ \mathcal{R}_w^{\nu}f$ belongs to $L^2_{\alpha,\beta} (D^2)$ and is clearly holomorphic on $D^2$ by Stone-Weierstrass theorem. Hence, 
$\mathcal{R}_w^{\nu}(L^2_{\nu}(\C)) \subset B^2_{\alpha,\beta} (D^2)$. 
\end{proof}

\begin{remark}
	Concerning the converse inclusion, we can provide an explicit example showing that the range of $L^2_{\nu}(\C)$ by $\mathcal{R}_w^{\nu}$ is strictly contained in $ B^2_{\alpha,\beta} (D^2)$. However, this can be reproved using compactness (discussed below) of the transform $\mathcal{R}_w^{\nu}$, since the range of compact operator is not closed in $ B^2_{\alpha,\beta} (D^2)$ ay least for in $\alpha,\beta>0$.
\end{remark}

\begin{proposition}\label{propnull}
%	If $w \notin \mathcal{Z}(H)$, the integral transform
%	$\mathcal{R}_w^{\nu}$ is one-to-one, while when $w\in \mathcal{Z}(H)$, the null space of $\mathcal{R}_w^{\nu}$ acting on $L^2_{\nu}(\C)$ is given by 
%	$$Ker(\mathcal{R}_w^{\nu}) =  Span\{H_{m,n}; \, (m,n)\in N_{w}(H) \}.$$ 
The null space of $\mathcal{R}_w^{\nu}$ acting on $L^2_{\nu}(\C)$ is a vector space characterized explicitly as  
$$Ker(\mathcal{R}_w^{\nu}) =  Span\{H_{m,n}; \, (m,n)\in N_{w}(H) \}$$
with dimension equals to the cardinal of $N_{w}(H) $.
Thus, the integral transform
$\mathcal{R}_w^{\nu}$ is one-to-one if and only if $w \notin \mathcal{Z}(H)$. 
\end{proposition}

\begin{proof} 
	According to \eqref{normRf}, if $f=\sum\limits_{m,n=0}^\infty\alpha_{m,n} \psi^\nu_{m,n}\in L^2_{\nu}(\C)$ is in the null space of $\mathcal{R}_w^{\nu}$, then $\norm{\mathcal{R}_w^{\nu}f}=0$ and hence
	$ \alpha_{m,n} \psi^\nu_{m,n}(w) =0$ for every $m,n$. Therefore, the null space of $\mathcal{R}_w^{\nu}$ reduces to $f=0$ when  $w \notin \mathcal{Z}(H)$. Now, for $w\in \mathcal{Z}(H)$, we conclude that
	$\alpha_{m,n} =0$
	for all $(m,n)\notin N_{w}(H)$. Therefore, $f= \sum\limits_{(m,n)\in N_w(H)} \alpha_{m,n} \psi^\nu_{m,n} $ which proves
	$$Ker (\mathcal{R}_w^{\nu})   \subset Span\{ \psi^\nu_{m,n}; \, (m,n)\in N_{w}(H) \}.$$  
	The converse inclusion is trivial and hence $\dim(Ker (\mathcal{R}_w^{\nu})) = Cardinal (N(H))$.	  
\end{proof}

%\newpage

\section{Proof of Theorem \ref{MThm2}: Compactness and membership in $p$-Schatten class}

%\begin{proof}
Set 
	$$ c^{\nu,\alpha,\beta}_{m,n}(w) := 
	%\psi^\nu_{m,n}(w)  \norm{e_{m,n}}_{\alpha,\beta}
	 \psi^\nu_{m,n}(w)  \left(\gamma^{\alpha,\beta}_{m,n} \right)^{1/2} $$ 
%we can rewrite $\mathcal{R}_w^{\nu}f$ in \eqref{ExprRf} as 
%\begin{align}\label{ExprRf}
%\mathcal{R}_w^{\nu}f 
%%= \sum\limits_{m,n=0}^\infty\psi^\nu_{m,n}(w)  \left(\gamma^{\alpha,\beta}_{m,n} \right)^{1/2} \scal{ f,\psi^\nu_{m,n} }_{L^2_{\nu}(\C)} \varphi^{\alpha,\beta}_{m,n} . 
%= \sum\limits_{m,n=0}^\infty
%c^{\nu,\alpha,\beta}_{m,n}(w) \scal{ f,\psi^\nu_{m,n} }_{L^2_{\nu}(\C)} 
%\varphi^{\alpha,\beta}_{m,n} 
%\end{align}	
and consider the finite rank operators 
	\begin{align*}
\mathcal{R}_{p,q}f=\sum\limits_{m=0}^{p} \sum\limits_{n=0}^{q} c^{\nu,\alpha,\beta}_{m,n}(w) \scal{ f,\psi^\nu_{m,n} }_{L^2_{\nu}(\C)}  	  \varphi^{\alpha,\beta}_{m,n}
	\end{align*} 
which are bounded and compact. Then, using the fact $ |c^{\nu,\alpha,\beta}_{m,n}(w)|^2 \leq e^{ \nu|w|^2 }\gamma^{\alpha,\beta}_{m,n}  $ as well as $\norm{f}_{L^2_{\nu}(\C)}^2 = \sum\limits_{m,n=0}^{\infty}  | \scal{ f,\psi^\nu_{m,n} }_{L^2_{\nu}(\C)} |^2$, we obtain
\begin{align*} \norm{(\mathcal{R}_w^{\nu}  - \mathcal{R}_{p,q} ) f}^2& = 
\sum\limits_{m=p+1}^{\infty} \sum\limits_{n=q+1}^{\infty} \left|c^{\nu,\alpha,\beta}_{m,n}(w) \right|^2 \left|\scal{ f,\psi^\nu_{m,n} }_{L^2_{\nu}(\C)} \right|^2\\
%&\leq 
%\left( 	\sum\limits_{m=p+1}^{\infty} \sum\limits_{n=q+1}^{\infty} |c^{\nu,\alpha,\beta}_{m,n}(w) |^2 \right) \left( \sum\limits_{m=p+1}^{\infty} \sum\limits_{n=q+1}^{\infty} |\scal{ f,\psi^\nu_{m,n} }_{L^2_{\nu}(\C)} |^2 \right) \\
&\leq 
\left( 	\sum\limits_{m=p+1}^{\infty} \sum\limits_{n=q+1}^{\infty} |c^{\nu,\alpha,\beta}_{m,n}(w) |^2 \right) \norm{f}_{L^2_{\nu}(\C)}^2 
\\& \leq  e^{ \nu|w|^2 } \left( \sum\limits_{m=p+1}^{\infty} \sum\limits_{n=q+1}^{\infty}  \gamma^{\alpha,\beta}_{m,n} \right)  \norm{f}_{L^2_{\nu}(\C)}^2 ,
\end{align*}
so that the following estimate for the operator norm
	\begin{align*}
\norm{\mathcal{R}_w^{\nu} - \mathcal{R}_{p,q} }^2  
\leq e^{ \nu|w|^2 } \sum\limits_{m=p+1}^{\infty} \sum\limits_{n=q+1}^{\infty}  \gamma^{\alpha,\beta}_{m,n}  .
\end{align*}
follows. 
%	\begin{align*}
%	||(\mathcal{R}_{p,q}-\mathcal{R}_w^{\nu})(f)||&=\norm{\sum\limits_{m>p,n>q}\left\langle f,\frac{H^{\nu}_{m,n}}{||H^{\nu}_{m,n}||}\right\rangle\frac{H^\nu_{m,n}(w,\overline{w})}{||H^{\nu}_{m,n}||}e_{m,n}}_{L^2_{\nu}(\C)} 
%	\end{align*}
%	using the 
The series in the right-hand side is convergent for $\alpha>0$ and $\beta>0$, and hence its rest goes to zero, so that $\lim\limits_{p,q\longrightarrow \infty } \norm{\mathcal{R}_{p,q}-\mathcal{R}_w^{\nu}}=0$.
%\end{proof}

The above discussion can be reformulated as follows.

\begin{proposition} 
	Let $\alpha>0$ and $\beta>0$, then  $\mathcal{R}_w^{\nu}(f)$ is compact. 
\end{proposition}

\begin{remark}
	The expansion in \eqref{ExprRf}, % can be rewritten as 
	\begin{align*}
	\mathcal{R}_w^{\nu}(f) =  \sum\limits_{m,n=0}^\infty c^{\nu,\alpha,\beta}_{m,n}(w) 
	\scal{ f,\psi^\nu_{m,n} }_{L^2_{\nu}(\C)}  	  \varphi^{\alpha,\beta}_{m,n},
	\end{align*}
	looks like the spectral decomposition of the operator $\mathcal{R}_w^{\nu}$. 
	% since  $c^{\nu,\alpha,\beta}_{m,n}(w)$ can be shown to verify	$\lim\limits_{m,n\longrightarrow\infty} c^{\nu,\alpha,\beta}_{m,n}(w)=0 $.
	%		since $\alpha,\beta>$ and 
%	\begin{align*} c^{\nu,\alpha,\beta}_{m,n}(w)	
%	 &\leq \pi e^{\frac{\nu|w|^2}{2}}\left(\frac{m!n! \Gamma(\alpha+1) \Gamma(\beta+1)}
%	 {\Gamma(m+\alpha+2)\Gamma(n+\beta+2)}\right)^{\frac{1}{2}}
%	 \\& \sim \pi \left(\frac{ \Gamma(\alpha+1) \Gamma(\beta+1)}{m^{\alpha+1}n^{\beta+1}}\right)^{\frac{1}{2}}
%	 \end{align*}
	%	 for $m,n$ large enough.
\end{remark}

From general context, we know that the adjoint of the integral transform $\mathcal{R}_w^{\nu}$ is given through 
$$ (\mathcal{R}_w^{\nu})^*g(z) =\scal{g, K^\nu_{u,v}( \cdot ; w ) }_{\alpha,\beta}  = \int_{D^2} g(u,v) \overline{K^\nu_{u,v}( z ; w )} d\mu_{\alpha,\beta} $$
for every function $g\in L^2_{\alpha,\beta} (D^2)$. 
This can easily be handled by direct computation. 
Subsequently, $(\mathcal{R}_w^{\nu})^*\mathcal{R}_w^{\nu}$ is an integral transform on $L^2_{\nu}(\C)$
$$ (\mathcal{R}_w^{\nu})^*\mathcal{R}_w^{\nu} f = \int_{\C} S_w^\nu(\zeta,z) f(\zeta) e^{-\nu|\zeta|^2} d\lambda(\zeta) 
$$
 with kernel function given by 
$$ S_w^\nu(\zeta,z) := \sum_{m,n=0}^\infty
 | c^{\nu,\alpha,\beta}_{m,n}(w)|^2   
 \psi^\nu_{m,n}(z) \overline{\psi^\nu_{m,n}(\zeta)} .$$ 
Therefore,  we have 
\begin{align*}
(\mathcal{R}_w^{\nu})^*\mathcal{R}_w^{\nu}(f)=\sum\limits_{m,n=0}^\infty
| c^{\nu,\alpha,\beta}_{m,n}(w)|^2  \scal{ f,  \psi^\nu_{m,n}}\psi^\nu_{m,n},
\end{align*}
and in particular 
 \begin{align*}
(\mathcal{R}_w^{\nu})^*\mathcal{R}_w^{\nu}(\psi^\nu_{m,n})= 
| c^{\nu,\alpha,\beta}_{m,n}(w)|^2  \psi^\nu_{m,n}.
\end{align*}
Accordingly, since $\psi^\nu_{m,n}$ form an orthonormal basis of $L^2_{\nu}(\C)$, the  singular values of $\mathcal{R}_w^{\nu}$ which are the eigenvalue of  $|\mathcal{R}_w^{\nu}|:= ( (\mathcal{R}_w^{\nu})^* \mathcal{R}_w^{\nu})^{1/2}$ are given by 
$$s^{\nu,\alpha,\beta}_{m,n}(w)=|c^{\nu,\alpha,\beta}_{m,n}(w)|
 :=|\psi^\nu_{m,n}(w)|\left(\gamma^{\alpha,\beta}_{m,n} \right)^{1/2} .$$
  More explicitly,  
%$$s^{\nu,\alpha,\beta}_{m,n}(w)=
% \left(  \frac{\nu \pi^2 \Gamma(\alpha+1) \Gamma(\beta+1) m!n!}{\pi\nu^{m+n} m!n!\Gamma(\alpha+m+2)\Gamma(\beta+n+2)} \right)^{1/2} |H^{\nu}_{m,n}(w,\bw)|
% .$$
 \begin{align}\label{explicitsinfval}
 s^{\nu,\alpha,\beta}_{m,n}(w)=
\left(  \frac{\nu \pi \Gamma(\alpha+1) \Gamma(\beta+1) }{\nu^{m+n} \Gamma(\alpha+m+2)\Gamma(\beta+n+2)} \right)^{1/2} |H^{\nu}_{m,n}(w,\bw)|
.
\end{align}
Subsequently,
 \begin{align*}
 s^{\nu,\alpha,\beta}_{m,n}(w)	
 &\leq \pi e^{\frac{\nu|w|^2}{2}}\left(\frac{m!n! \Gamma(\alpha+1) \Gamma(\beta+1)}
 {\Gamma(m+\alpha+2)\Gamma(n+\beta+2)}\right)^{\frac{1}{2}}.
 %\\& \sim \pi \left(\frac{ \Gamma(\alpha+1) \Gamma(\beta+1)}{m^{\alpha+1}n^{\beta+1}}\right)^{\frac{1}{2}}
 \end{align*}
% for $m,n$ large enough. 
 It follows that  
$\lim\limits_{m,n\longrightarrow\infty} s^{\nu,\alpha,\beta}_{m,n}(w)=0 $ since $\alpha,\beta>0$ and the right hand-side behaves as $ m^{-\alpha-1}n^{-\beta-1}$ for $m,n$ large enough. 
 Moreover,  $\mathcal{R}_{w}^{\nu}$ is in the $p$-Schatten class if  $p(\alpha+1)/2>1$ and $p(\beta+1)/2>1$, i.e. such that 
 $p>\max(2/(\alpha+1); 2/(\beta+1))$. This readily follows by means of 
\begin{align*}
(s^{\nu,\alpha,\beta}_{m,n}(w))^p
& \leq \pi^p  \frac{\left( e^{\nu |w|^2 } \Gamma(\alpha+1) \Gamma(\beta+1)\right)^{\frac{p}{2}}}{m^{p(\alpha+1)/2}n^{p(\beta+1)/2}}.
\end{align*} 
%\textcolor{red}{For the converse, if  $\mathcal{R}_{w}^{\nu}$ is in the $p$-Schatten class, then }.
%\newpage
Thus we have proved the following 

\begin{proposition}
Let $\alpha,\beta >0$. 
The singular values of  $\mathcal{R}_{w}^{\nu}$  are given by \eqref{explicitsinfval}. Moreover, $\mathcal{R}_{w}^{\nu}$ is in $p$-Schatten class for every $p$ such that  $p>\max(2/(\alpha+1); 2/(\beta+1))$.
\end{proposition}

\begin{remark}
 $\mathcal{R}_{w}^{\nu}$ is not a trace class operator if  
$\alpha\leq 1$ or $\beta\leq 1$. However, it is always a Hilbert--Schmidt operator for $\alpha,\beta>0$. 
\end{remark} 

%\newpage

\section{Connection to generalized Fractional Hankel transform}
The transform  $ \psi \longmapsto \mathcal{R}_{w}^{\nu}\psi(u,v)$, for fixed $u,v\in D$,  seen as function in the variable $w$, is exactly the non trivial $2$d fractional Fourier transform considered by Zayed in \cite{Zayed2018}, to wit
\begin{align*}
\mathcal{F}^{\nu}_{u,v}\psi (\xi)
= \int_{\C} \psi(\zeta) K^\nu_{u,v}( \zeta ; \xi) e^{-\nu |\zeta|^2} d\lambda(\zeta)  = \mathcal{R}_{\xi}^{\nu}\psi (u,v),
\end{align*}
where $K^\nu_{u,v}( \zeta ; \xi )$ is as in \eqref{2dFrFTKernel2}.
The eigenfunctions of $\mathcal{F}^{\nu}_{u,v}$ are the It\^o--Hermite polynomials, 
$$ \mathcal{F}^{\nu}_{u,v}  \psi^\nu_{m,n}= u^mv^n \psi^\nu_{m,n}.$$

The connection of fractional Fourier coefficients of given  $f\in L^2_{\nu}(\C)$ to its fractional Fourier transform $\mathcal{F}^{\nu}_{u,v}f$  is given in Theorem \ref{MThm3}   by means of the fractional Hankel transform in \eqref{GFrHT}
%\begin{align}\label{GFrHT}
%\mathcal{H}_{u,v}^{\nu,\alpha}(f)(y)&=
%\frac{2\nu  }{1-uv} \left( \frac{u}{v}\right)^{\alpha/2}   \int_0^\infty x f(x) I_{\alpha} \left( \frac{2\nu \sqrt{uv }}{1-uv}  xy \right)  e^{\frac{-\nu (x^2  +uv y^2)}{1-uv} } dr.
%\end{align}
which is a specific  generalization of the classical Hankel transform \cite{}.
%$$ \mathcal{H}_{m}(\rho) = \int_0^\infty r f(r) J_m(\rho r) dr.$$
 To this end, we begin by
 % writing $\mathcal{F}^{\nu}_{u,v}$ in terms of  $ \mathcal{H}_{u,v}^{\nu,\alpha}$ for rotational symmetric functions.
interpreting $\mathcal{H}_{u,v}^{\nu,\alpha}$  for integer order $\alpha=k$ as
the radial part of the $2$d Fourier transform $\mathcal{F}^{\nu}_{u,v}$ of $k$-rotationally symmetric function.

\begin{proposition}\label{propFHH} For every rotational  $\psi_k(\zeta) = \Psi(|\zeta|) e^{ik\theta}$, we have 
%	\begin{align*}
%	\mathcal{F}^{\nu}_{u,v}\psi (\xi)
%	&=   
%	\frac{2\nu (-i)^{k} }{1-uv} \left( \frac{ u\xi}{v\bxi}\right) ^{k/2}  
%	\int_0^\infty r\Psi(r) J_{k} \left( \frac{2i\nu \sqrt{uv }} {1-uv}|\xi| r \right)  e^{\frac{-\nu (r^2+uv|\xi|^2)}{1-uv}  } dr .
%	\end{align*}
		\begin{align*}
	\mathcal{F}^{\nu}_{u,v}\psi_k (\xi) = \left( \frac{\xi}{\bxi}\right) ^{k/2} \mathcal{H}_{u,v}^{\nu,\alpha}(\Psi)(|\xi|)  .
	\end{align*}
\end{proposition}

\begin{proof}
By expanding the kernel function $K^\nu_{u,v}( \zeta ; \xi) e^{-\nu |\zeta|^2}$ in power series as 
	\begin{align*}
 K^\nu_{u,v}( \zeta ; \xi) e^{-\nu |\zeta|^2} 
 =
\exp\left(\frac{-\nu (|\zeta|^2+ uv|\xi|^2) }{1-uv}    \right)
\sum_{m,n=0}^\infty \frac{\nu^{m+n} (u\xi)^m (v\bxi)^n}{(1-uv)^{m+n}}    \frac{\zeta^n\bzeta^m}{m!n!} . 
\end{align*}
and using polar coordinates we see that $\mathcal{F}^{\nu}_{u,v}\psi (\xi)$ takes the form 
\begin{align*}
%\mathcal{F}^{\nu}_{u,v}\psi (\xi)
%&= \frac{\ell_{u,v} e^{-uv \ell_{u,v} |\xi|^2} }{\pi} 
%\sum_{m,n=0}^\infty \ell_{u,v}^{m+n} u^m v^n \xi^m \bxi^n   
%\int_0^{2\pi}\int_0^\infty \Psi(r)    
%e^{i(k+m-n)\theta}\frac{r^{m+n}}{m!n!} e^{-\ell_{u,v} r^2 } 
%rdrd\theta \\&=
 \frac{\ell_{u,v} }{\pi} 
\sum_{m,n=0}^\infty    
\int_0^{2\pi}\int_0^\infty \psi(re^{i\theta})    
e^{i(n-m)\theta}\frac{(u \xi)^m (v\bxi)^n(\ell_{u,v}r)^{m+n}}{m!n!} e^{-\ell_{u,v} (r^2 + uv |\xi|^2)} 
rdrd\theta 
\end{align*}
where $\ell_{u,v}$ stands for $\ell_{u,v}:= {\nu}/{(1-uv)}$. 
Therefore, for every rotational symmetric function  $\psi_k(\zeta) = \Psi(|\zeta|) e^{ik\theta}$, it reduces further to
\begin{align*}
%\mathcal{F}^{\nu}_{u,v}\psi (\xi)
%&= 2\ell_{u,v} e^{-uv \ell_{u,v} |\xi|^2}  \sum_{m,n=0}^\infty \ell_{u,v}^{m+n} u^m v^n \xi^m \bxi^n   \delta_{k+m,n} \int_0^\infty \Psi(r^2)  \frac{r^{m+n}}{m!n!} e^{-\ell_{u,v} r^2 }  rdr \\&=
 2 \ell_{u,v}  
\left( \frac{ u\xi}{v\bxi}\right) ^{k/2}  
\int_0^\infty r\Psi(r)  \left(\sum_{m=0}^\infty 
\frac{(-i)^{2m+k}}{m!(k+m)!}  \left(i\ell_{u,v} \sqrt{uv} |\xi| r  \right)^{2m+k}\right) e^{-\ell_{u,v}(r^2  +uv |\xi|^2) } dr.
%\\
%\\&=  (-i)^{k}\ell_{u,v} e^{-uv \ell_{u,v} |\xi|^2} \left( \frac{ v\bxi}{  u\xi}\right) ^{k/2} \int_0^\infty \Psi(t)  e^{-\ell_{u,v} t } \sum_{m=0}^\infty \frac{(-1)^{m} \varepsilon_{n-k} }{m!(k+m)!} \left( i\ell_{u,v}\sqrt{uv t}|\xi|  \right)^{2m+k}  dt\\
%&=  2\ell_{u,v}(-i)^{k}  
% \left( \frac{u\xi}{  v\bxi}\right) ^{k/2}  \int_0^\infty r\Psi(r) 
%J_{k} \left( \frac{2i\nu \sqrt{uv }}{1-uv}|\xi| r \right)  e^{-\ell_{u,v} (r^2  +uv |\xi|^2) } dr .
\end{align*}
Hence, by means of \cite[p.222]{AndrewsAskeyRoy1999} 
\begin{align}\label{BesselFct} I_\alpha (\xi) :=    \sum_{n=0}^{\infty} \frac{1}{n! \Gamma(\alpha+n+1)} \left(  \frac{\xi}{2}\right)^{2n+\alpha}
\end{align}
and $I_\alpha (-\xi) = (-1)^\alpha I_\alpha (\xi) $, it follows 
\begin{align*}
\mathcal{F}^{\nu}_{u,v}\psi_k (\xi)
&=  2\ell_{u,v}   
\left( \frac{u\xi}{  v\bxi}\right) ^{k/2}  \int_0^\infty r\Psi(r) 
I_{k} \left( \frac{2\nu \sqrt{uv }}{1-uv}|\xi| r \right)  e^{-\ell_{u,v} (r^2  +uv |\xi|^2) } dr .
\end{align*}
% this yields 
%\begin{align*}
%\mathcal{F}^{\nu}_{u,v}\psi (\xi)
%&=   
%2(-1)^{k}\ell_{u,v} e^{-uv \ell_{u,v} |\xi|^2} \left( \frac{ -v\bxi}{  u\xi}\right) ^{k/2}  
%\int_0^\infty r\psi(re^{i\theta}) e^{-ik\theta} J_{k} \left(2i \ell_{u,v}\sqrt{uv }|\xi| r \right)  e^{-\ell_{u,v} r^2 } dr .
%\end{align*}
% for rotational symmtric function.
\end{proof}

\begin{proof}[Proof of Theorem \ref{MThm3}]
 Notice first that for arbitrary $f\in L^2_\nu(\C)$, we have 
 $$f(r e^{i\theta}) =\sum\limits_{k\in \mathbb{Z}}g_k(r)e^{in\theta}.$$
  Therefore, by setting $\xi= \rho e^{i\varphi}$  and making appeal of Proposition \ref{propFHH} we get  
\begin{align*}
 \mathcal{F}^{\nu}_{u,v}f(\xi) 
%&= \mathcal{F}^{\nu}_{u,v}\left( \sum\limits_{n\in \mathbb{Z}}g_k(r)e^{in\theta}\right) (\xi) \\
%&= \sum\limits_{k\in \mathbb{Z}} \mathcal{F}^{\nu}_{u,v}\left( g_k(r)e^{ik\theta}\right) (\xi) 
%\\&= \sum\limits_{k\in \mathbb{Z}}   \frac{2\nu (-1)^{k} }{1-uv}  \left( \frac{ -v\bxi}{  u\xi}\right) ^{k/2}  \int_0^\infty r g_k (r) J_{m} \left( \frac{2i\nu \sqrt{uv }}{1-uv}|\xi| r \right)  e^{\frac{-\nu (r^2  +uv |\xi|^2)}{1-uv} } dr \\&
= \sum\limits_{k\in \mathbb{Z}}   \left( \frac{\xi}{\bxi}\right) ^{k/2} \mathcal{H}_{u,v}^{\nu,\alpha}(g_k)(\rho)
%\\&= \sum\limits_{n\in \mathbb{Z}}  \left(  \frac{2\nu (-1)^{k} }{1-uv}   \left( \frac{ -v}{  u}\right) ^{k/2}   \int_0^\infty r g_k(r) J_{k} \left( \frac{2i\nu \sqrt{uv }}{1-uv}|\xi| r \right)  e^{\frac{-\nu (r^2  +uv |\xi|^2)}{1-uv} } dr \right) e^{ik\varphi}\\&
= \sum\limits_{k\in \mathbb{Z}}    \mathcal{H}_{u,v}^{\nu,\alpha}(g_k)(\rho) e^{ik\varphi}.
\end{align*}
Accordingly, by identification to
$\mathcal{F}^{\nu}_{u,v}f (\rho e^{i\varphi})=\sum\limits_{k\in \mathbb{Z}}G_k(\rho)e^{ik\varphi} $, we see that the Fourier coefficients
$G_k(\rho)$ and $g_k(r)$ 
%\begin{align*}
%G_m(\rho) & =\frac{1}{2\pi}\int_{0}^{2\pi}\mathcal{G}(\rho,\varphi)e^{-im\varphi}d\varphi
%\quad \mbox{and} \quad 
%g_n(r)=\frac{1}{2\pi}\int_{0}^{2\pi}g(r,\theta)e^{-in\theta}d\theta.
%\end{align*}
of $\theta \longmapsto \mathcal{F}^{\nu}_{u,v}f(\rho e^{i\theta})$ and $\theta \longmapsto f(r^{i\theta})$, respectively, satisfy 
%\begin{align*}
%G_m(\rho) =  2(-i)^{m}\ell_{u,v}  \left( \frac{v}{  u}\right) ^{m/2}   \int_0^\infty r g_{m}(r) J_{m} \left(2i \ell_{u,v}\sqrt{uv } \rho r \right)  e^{-\ell_{u,v} (r^2  +uv \rho^2)} dr 
%\end{align*}
%
%
%	The fractional Fourier coefficients $G_m$ and $g_m$ of given  $f\in L^2_{\nu}(\C)$ and its fractional Fourier transform $\mathcal{F}^{\nu}_{u,v}f$, respectively, are connected by 
\begin{align*}
G_k(\rho)
% &=  	\frac{2\nu (-i)^{k} }{1-uv} \left( \frac{ v}{u}\right)^{k/2}   \int_0^\infty r g_{k}(r) J_{k} \left( \frac{2i\nu \sqrt{uv }}{1-uv}  \rho r \right)  e^{\frac{-\nu (r^2  +uv \rho^2)}{1-uv} } dr \\&
=\mathcal{H}_{u,v}^{\nu,\alpha}(g_k)(\rho).
\end{align*}
\end{proof}

\end{document}